\documentclass{article}

\usepackage[english]{babel}

\usepackage[letterpaper,top=2cm,bottom=2cm,left=3cm,right=3cm,marginparwidth=1.75cm]{geometry}

\usepackage{amsmath}
\usepackage{amssymb}
\usepackage{amsthm}
\usepackage{tabstackengine}
\stackMath
\makeatletter
\renewcommand\TAB@delim[1]{\scriptstyle#1}
\makeatother
\setstackgap{S}{2pt}
\usepackage{xcolor}
\usepackage{graphicx}
\usepackage[colorlinks=true, allcolors=cyan]{hyperref}
\newtheorem{theorem}{Theorem}
\newtheorem{lemma}{Lemma}

\begin{document}
\begin{center}
		\vskip 1cm{\huge \bf Evaluating Six Apéry-like Series of Weight $5$}\\		
		\vskip 5mm
		Jorge Antonio González Layja\\
        Independent researcher\\
		México\\
		\noindent \href{mailto:jorgelayja16@gmail.com}{\texttt{jorgelayja16@gmail.com}}\\
		January 26, 2025
	\end{center}

\begin{abstract}
The main objective of this paper is to evaluate six new Apéry-like series of weight $5$ in closed form. These series involve harmonic numbers and exhibit the characteristic reciprocal central binomial coefficient structure. Generating functions for the inverse sine and identities related to harmonic numbers are used to link each series to a variety of integrals containing $\ln \left(\sin \left(x\right)\right)$ and $\ln \left(\cos \left(x\right)\right)$, which are evaluated using a range of methods and identities.
\end{abstract}

    
\section{Introduction and Preliminaries}
The central focus of this paper is to obtain the closed forms of the following series:
$$\sum _{k=1}^{\infty }\frac{4^kH_k^{\left(3\right)}}{k^2\binom{2k}{k}},\quad\sum _{k=1}^{\infty }\frac{4^kH_k^{\left(2\right)}}{k^3\binom{2k}{k}},\quad\sum _{k=1}^{\infty }\frac{4^kH_k^2}{k^3\binom{2k}{k}}, \quad\sum _{k=1}^{\infty }\frac{4^kH_k}{k^4\binom{2k}{k}},$$
$$\sum _{k=1}^{\infty }\frac{4^kH_kH_k^{\left(2\right)}}{k^2\binom{2k}{k}},\quad \operatorname{and}\quad\sum _{k=1}^{\infty }\frac{4^kH_k^3}{k^2\binom{2k}{k}},$$
where $\binom{2k}{k}$ is the $k$-th central binomial coefficient.
\\[1.5ex]
In these series, $H_k$ denotes the $k$-th harmonic number, defined as:
$$H_k=\sum _{n=1}^k\frac{1}{n},$$
and the $k$-th generalized harmonic number of order $m$, denoted by $H_k^{\left(m\right)}$, is defined as:
$$H_k^{\left(m\right)}=\sum _{n=1}^k\frac{1}{n^m},\quad m\in \mathbb{Z}^+.$$
An Apéry-like series involving harmonic numbers, of weight $m_1+m_2+\cdots +m_n+n$, is expressed by the summand:
$$\sum _{k=1}^{\infty }\frac{4^kH_k^{\left(m_1\right)}H_k^{\left(m_2\right)}\cdots H_k^{\left(m_n\right)}}{k^n\binom{2k}{k}},$$
where $m_1,m_2,\dots ,m_n$ and $n$ are positive integers. Given this, we determine that our desired series all possess weight $5$. Next, we will introduce definitions, including special functions and other key identities, that will aid us in performing calculations and understanding the results clearly.
\\[1.5ex]
The polygamma function of order $m$ is defined as:
\begin{equation*}
\psi ^{\left(m\right)}\left(z\right)=\frac{d^m}{dz^m}\psi ^{\left(0\right)}\left(z\right)=\frac{d^{m+1}}{dz^{m+1}}\ln \left(\Gamma \left(z\right)\right),\quad m\in \mathbb{Z}_{\ge 0}\wedge \operatorname{\mathfrak{R}} \left(z\right)>0,\tag{1.1}
\label{1.1}
\end{equation*}
where $\psi ^{\left(0\right)}\left(z\right)$ is the digamma function, and $\Gamma \left(z\right)$, the gamma function, is defined as:
$$\Gamma \left(z\right)=\int _0^{\infty }x^{z-1}e^{-x}\:dx,\quad\operatorname{\mathfrak{R}} \left(z\right)>0.$$
If $m\in \mathbb{Z}^+$ and $\operatorname{\mathfrak{R}} \left(z\right)>0$, then the following integral representation for the polygamma function holds:
\begin{equation*}
\psi ^{\left(m\right)}\left(z\right)=-\int _0^1\frac{x^{z-1}\ln ^m\left(x\right)}{1-x}\:dx.\tag{1.2}
\label{1.2}
\end{equation*}
Additionally, the following identity holds:
\begin{equation*}
\psi ^{\left(m\right)}\left(z+1\right)=\left(-1\right)^{m+1}m!\left(\zeta \left(m+1\right)-H_z^{\left(m+1\right)}\right),\quad \operatorname{\mathfrak{R}}\left(z\right)>-1,\tag{1.3}
\label{1.3}
\end{equation*}
where $\zeta \left(m+1\right)$ is the Riemann zeta function, defined as:
$$\zeta \left(m\right)=\sum _{k=1}^{\infty }\frac{1}{k^m},\quad \operatorname{\mathfrak{R}}\left(m\right)>1,$$
and $H_z^{\left(m+1\right)}$ is the analytic continuation of the generalized harmonic number, defined as:
$$H_z^{\left(m\right)}=\sum _{k=1}^{\infty }\left(\frac{1}{k^m}-\frac{1}{\left(k+z\right)^m}\right),\quad z\notin \mathbb{Z}^-.$$
The beta function is defined as:
$$\operatorname{B}\left(\alpha ,\beta \right)=\frac{\Gamma \left(\alpha \right)\Gamma \left(\beta \right)}{\Gamma \left(\alpha +\beta \right)}=\int _0^1x^{\alpha -1}\left(1-x\right)^{\beta -1}\:dx,\quad \operatorname{\mathfrak{R}}\left(\alpha \right)>0\wedge \operatorname{\mathfrak{R}}\left(\beta \right)>0,$$
alternatively expressed as:
\begin{equation*}
\operatorname{B}\left(\alpha ,\beta \right)=2\int _0^{\frac{\pi }{2}}\sin ^{2\alpha -1}\left(x\right)\cos ^{2\beta -1}\left(x\right)\:dx,\quad \operatorname{\mathfrak{R}}\left(\alpha \right)>0\wedge \operatorname{\mathfrak{R}}\left(\beta \right)>0.\tag{1.4}
\label{1.4}
\end{equation*}
The polylogarithm function of order $m$ is defined as:
$$\operatorname{Li}_m\left(z\right)=\sum _{k=1}^{\infty }\frac{z^k}{k^m},\quad \left|z\right|\le 1,$$
or, alternatively,
$$\operatorname{Li}_m\left(z\right)=\int _0^z\frac{\operatorname{Li}_{m-1}\left(x\right)}{x}\:dx,\quad z\in \mathbb{C},$$
where $\operatorname{Li}_1\left(z\right)=-\ln \left(1-z\right)$.
\\[1.5ex]
Lastly, the inverse tangent is defined in the following logarithmic form:
$$\arctan \left(z\right)=-\frac{i}{2}\ln \left(\frac{\left(1+iz\right)^2}{1+z^2}\right),\quad z\in \mathbb{C}\wedge z\neq \pm i,$$
which can also be rearranged as:
\begin{equation*}
\ln \left(1+iz\right)=\frac{1}{2}\ln \left(1+z^2\right)+i\arctan \left(z\right),\quad z\in \mathbb{C}\wedge z\neq \pm i.\tag{1.5}
\label{1.5}
\end{equation*}
\section{The Lemmas and their Proofs}
\begin{lemma} \label{lma1}The following equalities hold:
\begin{alignat*}{2}
\left(\operatorname{i}\right)&\quad&&\frac{1}{2}\sum _{k=1}^{\infty }\frac{4^kx^{2k-1}}{k\binom{2k}{k}}=\frac{\arcsin \left(x\right)}{\sqrt{1-x^2}},\quad\left|x\right|<1;\\
\left(\operatorname{ii}\right)&\quad&&\frac{1}{2}\sum _{k=1}^{\infty }\frac{4^kx^{2k}}{k^2\binom{2k}{k}}=\arcsin ^2\left(x\right),\quad\left|x\right|\le 1;\\
\left(\operatorname{iii}\right)&\quad&&\frac{3}{2}\sum _{k=1}^{\infty }\frac{4^kH_k^{\left(2\right)}}{k^2\binom{2k}{k}}x^{2k}-\frac{3}{2}\sum _{k=1}^{\infty }\frac{4^k}{k^4\binom{2k}{k}}x^{2k}=\arcsin ^4\left(x\right),\quad\left|x\right|\le 1.
\end{alignat*}
\begin{proof}
The identities in points $\left(\operatorname{i}\right)$ and $\left(\operatorname{ii}\right)$ are demonstrated in \cite[pp. 331-333]{book2}. Furthermore, the identity in point $\left(\operatorname{iii}\right)$ is established in \cite[pp. 107-109]{book1}.
\end{proof}
\end{lemma}

\begin{lemma} \label{lma2}Let k be a positive integer. Then, the following equalities hold:
\begin{alignat*}{2}
\left(\operatorname{i}\right)&\quad&&\int _0^1x^{k-1}\ln \left(1-x\right)\:dx=-\frac{H_k}{k};\\
\left(\operatorname{ii}\right)&\quad&&\int _0^1x^{k-1}\ln ^2\left(x\right)\ln \left(1-x\right)\:dx=-2\frac{H_k}{k^3}-2\frac{H_k^{\left(2\right)}}{k^2}-2\frac{H_k^{\left(3\right)}}{k}+2\zeta \left(2\right)\frac{1}{k^2}+2\zeta \left(3\right)\frac{1}{k};\\
\left(\operatorname{iii}\right)&\quad&&\int _0^1x^{k-1}\ln ^2\left(1-x\right)\:dx=\frac{H_k^2+H_k^{\left(2\right)}}{k};\\
\left(\operatorname{iv}\right)&\quad&&\int _0^1x^{k-1}\ln \left(x\right)\ln ^2\left(1-x\right)\:dx=-\frac{H_k^2}{k^2}-\frac{H_k^{\left(2\right)}}{k^2}-2\frac{H_k^{\left(3\right)}}{k}-2\frac{H_kH_k^{\left(2\right)}}{k}+2\zeta \left(2\right)\frac{H_k}{k}+2\zeta \left(3\right)\frac{1}{k};\\
\left(\operatorname{v}\right)&\quad&&\int _0^1x^{k-1}\ln ^3\left(1-x\right)\:dx=-\frac{H_k^3+3H_kH_k^{\left(2\right)}+2H_k^{\left(3\right)}}{k}.
\end{alignat*}
\begin{proof}
The integrals in points $\left(\operatorname{i}\right)$, $\left(\operatorname{iii}\right)$, and $\left(\operatorname{v}\right)$ are evaluated in \cite[pp. 59-62]{book2}. On the other hand, the integrals in points $\left(\operatorname{ii}\right)$ and $\left(\operatorname{iv}\right)$ are obtained by differentiating with respect to $k$ the integrals in points $\left(\operatorname{i}\right)$ and $\left(\operatorname{iii}\right)$, as demonstrated in \cite[pp. 222-227]{book3}.
\end{proof}
\end{lemma}

\begin{lemma} \label{lma3}The following equalities hold:
\begin{alignat*}{2}
\left(\operatorname{i}\right)&\quad&&-\ln \left(2\right)-\sum _{k=1}^{\infty }\frac{\cos \left(2kx\right)}{k}=\ln \left(\sin \left(x\right)\right),\quad 0<x<\pi ;\\
\left(\operatorname{ii}\right)&\quad&&-\ln \left(2\right)-\sum _{k=1}^{\infty }\frac{\left(-1\right)^k\cos \left(2kx\right)}{k}=\ln \left(\cos \left(x\right)\right),\quad \left|x\right|<\frac{\pi }{2};\\
\left(\operatorname{iii}\right)&\quad&&\int _0^{\frac{\pi }{2}}x\ln \left(\sin \left(x\right)\right)\:dx=\frac{7}{16}\zeta \left(3\right)-\frac{3}{4}\ln \left(2\right)\zeta \left(2\right);\\
\left(\operatorname{iv}\right)&\quad&&\int _0^{\frac{\pi }{2}}x\ln \left(\cos \left(x\right)\right)\:dx=-\frac{7}{16}\zeta \left(3\right)-\frac{3}{4}\ln \left(2\right)\zeta \left(2\right);\\
\left(\operatorname{v}\right)&\quad&&\int _0^{\frac{\pi }{2}}x^2\ln \left(\sin \left(x\right)\right)\:dx=\frac{3\pi }{16}\zeta \left(3\right)-\frac{\pi ^3}{24}\ln \left(2\right);\\
\left(\operatorname{vi}\right)&\quad&&\int _0^{\frac{\pi }{2}}x^2\ln \left(\cos \left(x\right)\right)\:dx=-\frac{\pi }{4}\zeta \left(3\right)-\frac{\pi ^3}{24}\ln \left(2\right);\\
\left(\operatorname{vii}\right)&\quad&&\int _0^{\frac{\pi }{2}}x^3\ln \left(\sin \left(x\right)\right)\:dx=-\frac{93}{128}\zeta \left(5\right)+\frac{27}{32}\zeta \left(2\right)\zeta \left(3\right)-\frac{45}{32}\ln \left(2\right)\zeta \left(4\right);\\
\left(\operatorname{viii}\right)&\quad&&\int _0^{\frac{\pi }{2}}x^3\ln \left(\cos \left(x\right)\right)\:dx=\frac{93}{128}\zeta \left(5\right)-\frac{9}{8}\zeta \left(2\right)\zeta \left(3\right)-\frac{45}{32}\ln \left(2\right)\zeta \left(4\right).
\end{alignat*}
\begin{proof}
The identities in points $\left(\operatorname{i}\right)$ and $\left(\operatorname{ii}\right)$ are proved in \cite[pp. 135-139]{book1}.
\\[1.5ex]
The integrals in points $\left(\operatorname{iii}\right)$ through $\left(\operatorname{viii}\right)$ are all proved in the same manner: by applying the appropriate expansion, integrating by parts, and observing that for $k\in \mathbb{N}$, it follows that $\sin \left(\pi k\right)=0$ and $\cos \left(\pi k\right)=\left(-1\right)^k$. This reduces each integral to expressions involving the Riemann zeta function.
\end{proof}
\end{lemma}

\begin{lemma} \label{lma4}The following equalities hold:
\begin{alignat*}{2}
\left(\operatorname{i}\right)&\quad&&\sum _{k=1}^{\infty }\frac{4^k}{k^2\binom{2k}{k}}=3\zeta \left(2\right);\\
\left(\operatorname{ii}\right)&\quad&&\sum _{k=1}^{\infty }\frac{4^k}{k^3\binom{2k}{k}}=-\frac{7}{2}\zeta \left(3\right)+6\ln \left(2\right)\zeta \left(2\right);\\
\left(\operatorname{iii}\right)&\quad&&\sum _{k=1}^{\infty }\frac{4^kH_k}{k^2\binom{2k}{k}}=\frac{7}{2}\zeta \left(3\right)+6\ln \left(2\right)\zeta \left(2\right);\\
\left(\operatorname{iv}\right)&\quad&&\int _0^{\frac{\pi }{2}}\cot \left(x\right)\ln ^2\left(\cos \left(x\right)\right)\:dx=\frac{1}{4}\zeta \left(3\right).
\end{alignat*}
\begin{proof}
We initiate with the series in point $\left(\operatorname{i}\right)$. To derive it, we simply use the result from point $\left(\operatorname{ii}\right)$ in Lemma \ref{lma1} and set $x=1$.
\\[1.5ex]
Next, we will use the same result for the series in point $\left(\operatorname{ii}\right)$. We have
$$\sum _{k=1}^{\infty }\frac{4^k}{k^3\binom{2k}{k}}=2\int _0^1\left(\sum _{k=1}^{\infty }\frac{4^kx^{2k}}{k^2\binom{2k}{k}}\right)\frac{1}{x}\:dx=4\int _0^1\frac{\arcsin ^2\left(x\right)}{x}\:dx.$$
Moreover, if we use integration by parts and apply the substitution $\arcsin \left(x\right)\mapsto x$, we get
$$\sum _{k=1}^{\infty }\frac{4^k}{k^3\binom{2k}{k}}=-8\int _0^{\frac{\pi }{2}}x\ln \left(\sin \left(x\right)\right)\:dx.$$
Therefore, by applying the result from point $\left(\operatorname{iii}\right)$ in Lemma \ref{lma3}, we obtain the desired closed form.
\\[1.5ex]
For the series in point $\left(\operatorname{iii}\right)$, by considering the result from point $\left(\operatorname{i}\right)$ in Lemma \ref{lma2}, it follows that
$$\sum _{k=1}^{\infty }\frac{4^kH_k}{k^2\binom{2k}{k}}=-\int _0^1\left(\sum _{k=1}^{\infty }\frac{4^kx^{k-1}}{k\binom{2k}{k}}\right)\ln \left(1-x\right)\:dx,$$
and by using the substitution $x\mapsto x^2$, applying the result from point $\left(\operatorname{i}\right)$ in Lemma \ref{lma1}, and then performing the substitution $\arcsin \left(x\right)\mapsto x$, we get
$$\sum _{k=1}^{\infty }\frac{4^kH_k}{k^2\binom{2k}{k}}=-8\int _0^{\frac{\pi }{2}}x\ln \left(\cos \left(x\right)\right)\:dx.$$
Upon employing the result from point $\left(\operatorname{iv}\right)$ in Lemma \ref{lma3}, the desired closed form follows.
\\[1.5ex]
For the integral in point $\left(\operatorname{iv}\right)$, applying the substitution $\cos \left(x\right)\mapsto \sqrt{x}$ leads to
$$\int _0^{\frac{\pi }{2}}\cot \left(x\right)\ln ^2\left(\cos \left(x\right)\right)\:dx=\frac{1}{8}\int _0^1\frac{\ln ^2\left(x\right)}{1-x}\:dx.$$
Thus, by applying the geometric series and integrating by parts to evaluate the resulting integral, the desired result is obtained.
\end{proof}
\end{lemma}

\begin{lemma} \label{lma5}The following equalities hold:
\begin{alignat*}{2}
\left(\operatorname{i}\right)&\quad&&\int _0^{\frac{\pi }{2}}\ln ^3\left(\sin \left(x\right)\right)\:dx=-\frac{3\pi }{4}\zeta \left(3\right)-\frac{\pi ^3}{8}\ln \left(2\right)-\frac{\pi }{2}\ln ^3\left(2\right);\\
\left(\operatorname{ii}\right)&\quad&&\int _0^{\frac{\pi }{2}}\ln \left(\sin \left(x\right)\right)\ln ^2\left(\cos \left(x\right)\right)\:dx=\frac{\pi }{8}\zeta \left(3\right)-\frac{\pi }{2}\ln ^3\left(2\right);\\
\left(\operatorname{iii}\right)&\quad&&\int _0^{\frac{\pi }{2}}x\ln ^3\left(\tan \left(x\right)\right)\:dx=\frac{93}{16}\zeta \left(5\right)+\frac{21}{16}\zeta \left(2\right)\zeta \left(3\right).
\end{alignat*}
\begin{proof}
To begin with the integral in point $\left(\operatorname{i}\right)$, we will use the alternative expression for the beta function in $\left(\ref{1.4}\right)$, where, by differentiating three times with respect to $\alpha $ and taking the limit as $\alpha,\beta  \rightarrow \frac{1}{2}$, we obtain
$$\int _0^{\frac{\pi }{2}}\ln ^3\left(\sin \left(x\right)\right)\:dx=\frac{1}{16}\lim_{\alignShortunderstack{\alpha\to&\frac{1}{2}\\\beta\to&\frac{1}{2}}}\frac{\partial ^3}{\partial \alpha ^3}\operatorname{B}\left(\alpha ,\beta \right)=\frac{1}{16}\lim_{\alignShortunderstack{\alpha\to&\frac{1}{2}\\\beta\to&\frac{1}{2}}}\frac{\partial ^3}{\partial \alpha ^3}\frac{\Gamma \left(\alpha \right)\Gamma \left(\beta \right)}{\Gamma \left(\alpha +\beta \right)},$$
and since $\frac{d}{d\alpha }\Gamma \left(\alpha \right)=\Gamma \left(\alpha \right)\psi ^{\left(0\right)}\left(\alpha \right)$, it follows that
$$\int _0^{\frac{\pi }{2}}\ln ^3\left(\sin \left(x\right)\right)\:dx=\frac{1}{16}\lim_{\alignShortunderstack{\alpha\to&\frac{1}{2}\\\beta\to&\frac{1}{2}}}\frac{\Gamma \left(\alpha \right)\Gamma \left(\beta \right)}{\Gamma \left(\alpha +\beta \right)}\left(\psi ^{\left(0\right)}\left(\alpha \right)-\psi ^{\left(0\right)}\left(\alpha +\beta \right)\right)^3$$
$$+\frac{3}{16}\lim_{\alignShortunderstack{\alpha\to&\frac{1}{2}\\\beta\to&\frac{1}{2}}}\frac{\Gamma \left(\alpha \right)\Gamma \left(\beta \right)}{\Gamma \left(\alpha +\beta \right)}\left(\psi ^{\left(0\right)}\left(\alpha \right)-\psi ^{\left(0\right)}\left(\alpha +\beta \right)\right)\left(\psi ^{\left(1\right)}\left(\alpha \right)-\psi ^{\left(1\right)}\left(\alpha +\beta \right)\right)$$
$$+\frac{1}{16}\lim_{\alignShortunderstack{\alpha\to&\frac{1}{2}\\\beta\to&\frac{1}{2}}}\frac{\Gamma \left(\alpha \right)\Gamma \left(\beta \right)}{\Gamma \left(\alpha +\beta \right)}\left(\psi ^{\left(2\right)}\left(\alpha \right)-\psi ^{\left(2\right)}\left(\alpha +\beta \right)\right).$$
Hence, by applying the limits and using the known values $\Gamma \left(\frac{1}{2}\right)=\sqrt{\pi }$, $\psi ^{\left(0\right)}\left(1\right)-\psi ^{\left(0\right)}\left(\frac{1}{2}\right)=2\ln \left(2\right)$, $\psi ^{\left(1\right)}\left(1\right)=\frac{1}{3}\psi ^{\left(1\right)}\left(\frac{1}{2}\right)=\zeta \left(2\right)=\frac{\pi ^2}{6}$, $\psi ^{\left(2\right)}\left(1\right)=\frac{1}{7}\psi ^{\left(2\right)}\left(\frac{1}{2}\right)=-2\zeta \left(3\right)$, we arrive at the closed form of the integral in point $\left(\operatorname{i}\right)$.
\\[1.5ex]
Similarly, for the integral in point $\left(\operatorname{ii}\right)$, we have
$$\int _0^{\frac{\pi }{2}}\ln \left(\sin \left(x\right)\right)\ln ^2\left(\cos \left(x\right)\right)\:dx=\frac{1}{16}\lim_{\alignShortunderstack{\alpha\to&\frac{1}{2}\\\beta\to&\frac{1}{2}}}\frac{\partial ^3}{\partial \alpha \partial \beta ^2}\operatorname{B}\left(\alpha ,\beta \right)=\frac{1}{16}\lim_{\alignShortunderstack{\alpha\to&\frac{1}{2}\\\beta\to&\frac{1}{2}}}\frac{\partial ^3}{\partial \alpha \partial \beta ^2}\frac{\Gamma \left(\alpha \right)\Gamma \left(\beta \right)}{\Gamma \left(\alpha +\beta \right)},$$
and by proceeding as shown previously, we find the desired closed form.
\\[1.5ex]
For the integral in point $\left(\operatorname{iii}\right)$, by applying the substitution $\tan \left(x\right)\mapsto x$, we get
$$\int _0^{\frac{\pi }{2}}x\ln ^3\left(\tan \left(x\right)\right)\:dx=\int _0^{\infty }\frac{\ln ^3\left(x\right)\arctan \left(x\right)}{1+x^2}\:dx.$$
Moreover, if we use that $\int _0^1\frac{x}{1+x^2t^2}\:dt=\arctan \left(x\right)$ and employ the substitution $x^2\mapsto x$, it follows that
\begin{equation*}
\int _0^{\frac{\pi }{2}}x\ln ^3\left(\tan \left(x\right)\right)\:dx=\frac{1}{16}\int _0^1\int _0^{\infty }\frac{\ln ^3\left(x\right)}{\left(1+x\right)\left(1+t^2x\right)}\:dx\:dt.\tag{2.1}
\label{2.x20}
\end{equation*}
Now, let us focus on the inner integral. Upon performing the substitution $xt\mapsto x$, we have
$$\int _0^{\infty }\frac{\ln ^3\left(x\right)}{\left(1+x\right)\left(1+t^2x\right)}\:dx=\int _0^{\infty }\frac{\ln ^3\left(\frac{x}{t}\right)}{\left(t+x\right)\left(1+tx\right)}\:dx\overset{\frac{1}{x}\mapsto x}{=}-\int _0^{\infty }\frac{\ln ^3\left(xt\right)}{\left(t+x\right)\left(1+xt\right)}\:dx,$$
which implies that
$$\int _0^{\infty }\frac{\ln ^3\left(x\right)}{\left(1+x\right)\left(1+t^2x\right)}\:dx=-3\ln \left(t\right)\int _0^{\infty }\frac{\ln ^2\left(x\right)}{\left(t+x\right)\left(1+tx\right)}\:dx-\int _0^{\infty }\frac{\ln ^3\left(t\right)}{\left(t+x\right)\left(1+tx\right)}\:dx.$$
By splitting both integrals at $x=1$, we obtain
$$\int _0^{\infty }\frac{\ln ^3\left(x\right)}{\left(1+x\right)\left(1+t^2x\right)}\:dx=-3\ln \left(t\right)\int _0^1\frac{\ln ^2\left(x\right)}{\left(t+x\right)\left(1+tx\right)}\:dx-3\ln \left(t\right)\int _1^{\infty }\frac{\ln ^2\left(x\right)}{\left(t+x\right)\left(1+tx\right)}\:dx$$
$$-\int _0^1\frac{\ln ^3\left(t\right)}{\left(t+x\right)\left(1+tx\right)}\:dx-\int _1^{\infty }\frac{\ln ^3\left(t\right)}{\left(t+x\right)\left(1+tx\right)}\:dx,$$
and if we use the substitution $\frac{1}{x}\mapsto x$ on the second and fourth resulting integrals, we arrive at
$$\int _0^{\infty }\frac{\ln ^3\left(x\right)}{\left(1+x\right)\left(1+t^2x\right)}\:dx=-6\ln \left(t\right)\int _0^1\frac{\ln ^2\left(x\right)}{\left(t+x\right)\left(1+tx\right)}\:dx-2\int _0^1\frac{\ln ^3\left(t\right)}{\left(t+x\right)\left(1+tx\right)}\:dx$$
$$=6\frac{\ln \left(t\right)}{1-t^2}\left(\int _0^1\frac{t\ln ^2\left(x\right)}{1+tx}\:dx-\int _0^1\frac{\frac{1}{t}\ln ^2\left(x\right)}{1+\frac{x}{t}}\:dx\right)+2\frac{\ln ^3\left(t\right)}{1-t^2}\left(\int _0^1\frac{t}{1+tx}\:dx-\int _0^1\frac{1}{t+x}\:dx\right)$$
$$=-12\frac{\ln \left(t\right)}{1-t^2}\left(\operatorname{Li}_3\left(-t\right)-\operatorname{Li}_3\left(-\frac{1}{t}\right)\right)+2\frac{\ln ^4\left(t\right)}{1-t^2},$$
where we used the result $\int _0^1\frac{t\ln ^2\left(x\right)}{1-tx}\:dx=2\operatorname{Li}_3\left(t\right)$, which can be derived using integration by parts. Additionally, using the identity for the trilogarithm $\operatorname{Li}_3\left(-t\right)-\operatorname{Li}_3\left(-\frac{1}{t}\right)=-\frac{1}{6}\ln ^3\left(t\right)-\zeta \left(2\right)\ln \left(t\right)$ leads to
$$\int _0^{\infty }\frac{\ln ^3\left(x\right)}{\left(1+x\right)\left(1+t^2x\right)}\:dx=4\frac{\ln ^4\left(t\right)}{1-t^2}+12\zeta \left(2\right)\frac{\ln ^2\left(t\right)}{1-t^2}.$$
Therefore, by replacing this into $\left(\ref{2.x20}\right)$, we obtain
$$\int _0^{\frac{\pi }{2}}x\ln ^3\left(\tan \left(x\right)\right)\:dx=\frac{1}{4}\int _0^1\frac{\ln ^4\left(t\right)}{1-t^2}\:dt+\frac{3}{4}\zeta \left(2\right)\int _0^1\frac{\ln ^2\left(t\right)}{1-t^2}\:dt.$$
Thus, by applying the geometric series and integration by parts to the resulting integrals, the closed form of the integral in point $\left(\operatorname{iii}\right)$ follows.
\end{proof}
\end{lemma}

\begin{lemma} \label{lma6}The following equalities hold:
\begin{alignat*}{2}
\left(\operatorname{i}\right)&\quad&&\int _0^{\frac{\pi }{2}}x^2\csc ^2\left(x\right)\ln ^2\left(\cos \left(x\right)\right)\:dx=\frac{\pi }{8}\zeta \left(3\right)+\frac{\pi ^3}{6}\ln \left(2\right)+\frac{\pi }{3}\ln ^3\left(2\right);\\
\left(\operatorname{ii}\right)&\quad&&\int _0^{\frac{\pi }{2}}x\cot \left(x\right)\ln ^2\left(\cos \left(x\right)\right)\:dx=-\frac{3\pi }{16}\zeta \left(3\right)+\frac{\pi ^3}{24}\ln \left(2\right)+\frac{\pi }{6}\ln ^3\left(2\right).
\end{alignat*}
\begin{proof}
We commence with the integral in point $\left(\operatorname{i}\right)$. Applying the substitution $x\mapsto \arctan \left(x\right)$ yields
\begin{equation*}
\int _0^{\frac{\pi }{2}}x^2\csc ^2\left(x\right)\ln ^2\left(\cos \left(x\right)\right)\:dx=\frac{1}{4}\int _0^{\infty }\frac{\arctan ^2\left(x\right)\ln ^2\left(1+x^2\right)}{x^2}\:dx.\tag{2.2}
\label{2.x14}
\end{equation*}
For the resulting integral,  we recall the identity in $\left(\ref{1.5}\right)$. By replacing $z$ with $x$, raising both sides to the power of $4$, and taking the real part, we get
\begin{equation*}
\operatorname{\mathfrak{R}} \left\{\ln ^4\left(1+ix\right)\right\}=\frac{1}{16}\ln ^4\left(1+x^2\right)-\frac{3}{2}\arctan ^2\left(x\right)\ln ^2\left(1+x^2\right)+\arctan ^4\left(x\right),\quad x\in \mathbb{R},\tag{2.3}
\label{2.x15}
\end{equation*}
and by isolating $\arctan ^2\left(x\right)\ln ^2\left(1+x^2\right)$, substituting it into $\left(\ref{2.x14}\right)$, and expanding, we arrive at
$$\int _0^{\frac{\pi }{2}}x^2\csc ^2\left(x\right)\ln ^2\left(\cos \left(x\right)\right)\:dx$$
\begin{equation*}
=\frac{1}{96}\int _0^{\infty }\frac{\ln ^4\left(1+x^2\right)}{x^2}\:dx+\frac{1}{6}\int _0^{\infty }\frac{\arctan ^4\left(x\right)}{x^2}\:dx-\frac{1}{6}\operatorname{\mathfrak{R}} \left\{\int _0^{\infty }\frac{\ln ^4\left(1+ix\right)}{x^2}\:dx\right\}.\tag{2.4}
\label{2.x16}
\end{equation*}
Let us now turn to the resulting expressions. For the first, we apply integration by parts, followed by the substitutions $x\mapsto \tan \left(x\right)$ and $x\mapsto \frac{\pi }{2}-x$, to obtain
$$\int _0^{\infty }\frac{\ln ^4\left(1+x^2\right)}{x^2}\:dx=-64\int _0^{\frac{\pi }{2}}\ln ^3\left(\sin \left(x\right)\right)\:dx.$$
Therefore, if we use the result from point $\left(\operatorname{i}\right)$ in Lemma \ref{lma5}, we immediately find
\begin{equation*}
\int _0^{\infty }\frac{\ln ^4\left(1+x^2\right)}{x^2}\:dx=48\pi \zeta \left(3\right)+8\pi ^3\ln \left(2\right)+32\pi \ln ^3\left(2\right).\tag{2.5}
\label{2.x17}
\end{equation*}
For the second, if we integrate by parts, apply the substitution $\arctan \left(x\right)\mapsto x$, and then integrate by parts again, we get
$$\int _0^{\infty }\frac{\arctan ^4\left(x\right)}{x^2}\:dx=-12\int _0^{\frac{\pi }{2}}x^2\ln \left(\sin \left(x\right)\right)\:dx,$$
and by using the result from point $\left(\operatorname{v}\right)$ in Lemma \ref{lma3}, it follows that
\begin{equation*}
\int _0^{\infty }\frac{\arctan ^4\left(x\right)}{x^2}\:dx=-\frac{9\pi }{4}\zeta \left(3\right)+\frac{\pi ^3}{2}\ln \left(2\right).\tag{2.6}
\label{2.x18}
\end{equation*}
For the third, let us first calculate the integral and then take the real part. Using the substitution $\frac{1}{1+ix}\mapsto x$, we have
$$\int _0^{\infty }\frac{\ln ^4\left(1+ix\right)}{x^2}\:dx=i\int _0^1\frac{\ln ^4\left(x\right)}{\left(1-x\right)^2}\:dx,$$
and if we employ the geometric series and integrate by parts, we obtain
$$\int _0^{\infty }\frac{\ln ^4\left(1+ix\right)}{x^2}\:dx=i\:24\zeta \left(4\right).$$
This implies that, by taking the real part, we arrive at
\begin{equation*}
\operatorname{\mathfrak{R}} \left\{\int _0^{\infty }\frac{\ln ^4\left(1+ix\right)}{x^2}\:dx\right\}=0.\tag{2.7}
\label{2.x19}
\end{equation*}
Thus, by substituting $\left(\ref{2.x17}\right)$, $\left(\ref{2.x18}\right)$, and $\left(\ref{2.x19}\right)$ into $\left(\ref{2.x16}\right)$, we find the desired result for the integral in point $\left(\operatorname{i}\right)$.
\\[1.5ex]
For the integral in point $\left(\operatorname{ii}\right)$, if we apply integration by parts, we obtain
$$\int _0^{\frac{\pi }{2}}x\cot \left(x\right)\ln ^2\left(\cos \left(x\right)\right)\:dx=\frac{1}{2}\int _0^{\frac{\pi }{2}}x^2\csc ^2\left(x\right)\ln ^2\left(\cos \left(x\right)\right)\:dx+\int _0^{\frac{\pi }{2}}x^2\ln \left(\cos \left(x\right)\right)\:dx.$$
Therefore, by using the previously calculated result from the integral in point $\left(\operatorname{i}\right)$ and the result from point $\left(\operatorname{vi}\right)$ in Lemma \ref{lma3}, we arrive at the desired closed form.
\end{proof}
\end{lemma}

\begin{lemma} \label{lma7}The following equalities hold:
\begin{alignat*}{1}
\left(\operatorname{i}\right)\quad&\int _0^{\frac{\pi }{2}}x^2\cot \left(x\right)\ln ^2\left(\cos \left(x\right)\right)\:dx\\
&=-\frac{217}{64}\zeta \left(5\right)-\frac{9}{16}\zeta \left(2\right)\zeta \left(3\right)+2\operatorname{Li}_5\left(\frac{1}{2}\right)+\frac{79}{16}\ln \left(2\right)\zeta \left(4\right)+\frac{2}{3}\ln ^3\left(2\right)\zeta \left(2\right)-\frac{1}{60}\ln ^5\left(2\right);\\
\left(\operatorname{ii}\right)\quad&\int _0^{\frac{\pi }{2}}x\ln ^3\left(\sin \left(x\right)\right)\:dx\\
&=-\frac{93}{128}\zeta \left(5\right)-\frac{9}{8}\zeta \left(2\right)\zeta \left(3\right)+3\operatorname{Li}_5\left(\frac{1}{2}\right)+\frac{57}{32}\ln \left(2\right)\zeta \left(4\right)-\frac{1}{2}\ln ^3\left(2\right)\zeta \left(2\right)-\frac{1}{40}\ln ^5\left(2\right).
\end{alignat*}
\begin{proof}
Let us start with the integral in point $\left(\operatorname{i}\right)$, where we use the substitution $x\mapsto \arctan \left(x\right)$ to obtain
\begin{equation*}
\int _0^{\frac{\pi }{2}}x^2\cot \left(x\right)\ln ^2\left(\cos \left(x\right)\right)\:dx=\frac{1}{4}\int _0^{\infty }\frac{\arctan ^2\left(x\right)\ln ^2\left(1+x^2\right)}{x\left(1+x^2\right)}\:dx.\tag{2.8}
\label{2.x4}
\end{equation*}
Now, we recall the identity in $\left(\ref{2.x15}\right)$. By isolating $\arctan ^2\left(x\right)\ln ^2\left(1+x^2\right)$, substituting it into $\left(\ref{2.x4}\right)$, and expanding, we arrive at
$$\int _0^{\frac{\pi }{2}}x^2\cot \left(x\right)\ln ^2\left(\cos \left(x\right)\right)\:dx$$
\begin{equation*}
=\frac{1}{96}\int _0^{\infty }\frac{\ln ^4\left(1+x^2\right)}{x\left(1+x^2\right)}\:dx+\frac{1}{6}\int _0^{\infty }\frac{\arctan ^4\left(x\right)}{x\left(1+x^2\right)}\:dx-\frac{1}{6}\operatorname{\mathfrak{R}} \left\{\int _0^{\infty }\frac{\ln ^4\left(1+ix\right)}{x\left(1+x^2\right)}\:dx\right\}.\tag{2.9}
\label{2.x5}
\end{equation*}
We now turn to the resulting integrals. By applying the substitution $\frac{1}{1+x^2}\mapsto x$ to the first integral, we have
$$\int _0^{\infty }\frac{\ln ^4\left(1+x^2\right)}{x\left(1+x^2\right)}\:dx=\frac{1}{2}\int _0^1\frac{\ln ^4\left(x\right)}{1-x}\:dx,$$
and by using the geometric series and integrating by parts, it follows that
\begin{equation*}
\int _0^{\infty }\frac{\ln ^4\left(1+x^2\right)}{x\left(1+x^2\right)}\:dx=12\zeta \left(5\right).\tag{2.10}
\label{2.x6}
\end{equation*}
Next, we evaluate the second integral. If we use the substitution $\arctan \left(x\right)\mapsto x$ and integrate by parts, we obtain
$$\int _0^{\infty }\frac{\arctan ^4\left(x\right)}{x\left(1+x^2\right)}\:dx=-4\int _0^{\frac{\pi }{2}}x^3\ln \left(\sin \left(x\right)\right)\:dx,$$
and by applying the result from point $\left(\operatorname{vii}\right)$ in Lemma \ref{lma3}, we arrive at
\begin{equation*}
\int _0^{\infty }\frac{\arctan ^4\left(x\right)}{x\left(1+x^2\right)}\:dx=\frac{93}{32}\zeta \left(5\right)-\frac{27}{8}\zeta \left(2\right)\zeta \left(3\right)+\frac{45}{8}\ln \left(2\right)\zeta \left(4\right).\tag{2.11}
\label{2.x7}
\end{equation*}
For the third expression, we will first calculate the integral and then take the real part. By splitting it at $x=1$, we get
$$\int _0^{\infty }\frac{\ln ^4\left(1+ix\right)}{x\left(1+x^2\right)}\:dx=\int _0^1\frac{\ln ^4\left(1+ix\right)}{x\left(1+x^2\right)}\:dx+\int _1^{\infty }\frac{\ln ^4\left(1+ix\right)}{x\left(1+x^2\right)}\:dx.$$
Subsequently, by applying the substitution $\frac{1}{1+ix}\mapsto x$ to both integrals, we obtain
$$\int _0^{\infty }\frac{\ln ^4\left(1+ix\right)}{x\left(1+x^2\right)}\:dx=-\int _0^{\frac{1-i}{2}}\frac{x\ln ^4\left(x\right)}{\left(1-x\right)\left(1-2x\right)}\:dx-\int _{\frac{1-i}{2}}^1\frac{x\ln ^4\left(x\right)}{\left(1-x\right)\left(1-2x\right)}\:dx.$$
Moreover, if we expand the integrals and rearrange, we arrive at
$$\int _0^{\infty }\frac{\ln ^4\left(1+ix\right)}{x\left(1+x^2\right)}\:dx=\int _0^1\frac{\ln ^4\left(x\right)}{1-x}\:dx-\int _0^{\frac{1-i}{2}}\frac{\ln ^4\left(x\right)}{1-2x}\:dx-\int _{\frac{1-i}{2}}^1\frac{\ln ^4\left(x\right)}{1-2x}\:dx,$$
where, by repeatedly applying integration by parts to the resulting expressions, most terms vanish, and we find
$$\int _0^{\infty }\frac{\ln ^4\left(1+ix\right)}{x\left(1+x^2\right)}\:dx=24\zeta \left(5\right)-12\operatorname{Li}_5\left(2\right).$$
Therefore, if we take the real part after using the identity $\operatorname{Li}_5\left(2\right)=\operatorname{Li}_5\left(\frac{1}{2}\right)+2\ln \left(2\right)\zeta \left(4\right)+\frac{1}{3}\ln ^3\left(2\right)\zeta \left(2\right)-\frac{1}{120}\ln ^5\left(2\right)-i\frac{\pi }{24}\ln ^4\left(2\right)$ (see \cite[(1.115), p. 47]{book1}), we arrive at
\begin{equation*}
\operatorname{\mathfrak{R}} \left\{\int _0^{\infty }\frac{\ln ^4\left(1+ix\right)}{x\left(1+x^2\right)}\:dx\right\}=24\zeta \left(5\right)-12\operatorname{Li}_5\left(\frac{1}{2}\right)-24\ln \left(2\right)\zeta \left(4\right)-4\ln ^3\left(2\right)\zeta \left(2\right)+\frac{1}{10}\ln ^5\left(2\right).\tag{2.12}
\label{2.x8}
\end{equation*}
Thus, by applying $\left(\ref{2.x6}\right)$, $\left(\ref{2.x7}\right)$, and $\left(\ref{2.x8}\right)$ to $\left(\ref{2.x5}\right)$, we obtain the desired result for the integral in point $\left(\operatorname{i}\right)$.
\\[1.5ex]
For the integral in point $\left(\operatorname{ii}\right)$, we first apply the change of variables $x\mapsto \frac{\pi }{2}-x$. This gives us
$$\int _0^{\frac{\pi }{2}}x\ln ^3\left(\sin \left(x\right)\right)\:dx=\frac{\pi }{2}\int _0^{\frac{\pi }{2}}\ln ^3\left(\cos \left(x\right)\right)\:dx-\int _0^{\frac{\pi }{2}}x\ln ^3\left(\cos \left(x\right)\right)\:dx,$$
and by applying the substitution $x\mapsto \frac{\pi }{2}-x$ to the first resulting integral, and $x\mapsto \arctan \left(x\right)$ to the second, we obtain
\begin{equation*}
\int _0^{\frac{\pi }{2}}x\ln ^3\left(\sin \left(x\right)\right)\:dx=\frac{\pi }{2}\int _0^{\frac{\pi }{2}}\ln ^3\left(\sin \left(x\right)\right)\:dx+\frac{1}{8}\int _0^{\infty }\frac{\arctan \left(x\right)\ln ^3\left(1+x^2\right)}{1+x^2}\:dx.\tag{2.13}
\label{2.x9}
\end{equation*}
Next, using the identity in $\left(\ref{1.5}\right)$, with $z$ replaced by $x$, raising both sides to the fourth power, and taking the imaginary part, we get
$$\operatorname{\mathfrak{I}} \left\{\ln ^4\left(1+ix\right)\right\}=\frac{1}{2}\arctan \left(x\right)\ln ^3\left(1+x^2\right)-2\arctan ^3\left(x\right)\ln \left(1+x^2\right),\quad x\in \mathbb{R}.$$
Moreover, by isolating $\arctan \left(x\right)\ln ^3\left(1+x^2\right)$, substituting it into $\left(\ref{2.x9}\right)$, and expanding, it follows that
$$\int _0^{\frac{\pi }{2}}x\ln ^3\left(\sin \left(x\right)\right)\:dx$$
\begin{equation*}
=\frac{\pi }{2}\int _0^{\frac{\pi }{2}}\ln ^3\left(\sin \left(x\right)\right)\:dx+\frac{1}{2}\int _0^{\infty }\frac{\arctan ^3\left(x\right)\ln \left(1+x^2\right)}{1+x^2}\:dx+\frac{1}{4}\operatorname{\mathfrak{I}} \left\{\int _0^{\infty }\frac{\ln ^4\left(1+ix\right)}{1+x^2}\:dx\right\}.\tag{2.14}
\label{2.x10}
\end{equation*}
Now, let us focus on the resulting integrals. For the first, we recall the result from point $\left(\operatorname{i}\right)$ in Lemma \ref{lma5}, from which we immediately have
\begin{equation*}
\frac{\pi }{2}\int _0^{\frac{\pi }{2}}\ln ^3\left(\sin \left(x\right)\right)\:dx=-\frac{9}{4}\zeta \left(2\right)\zeta \left(3\right)-\frac{45}{8}\ln \left(2\right)\zeta \left(4\right)-\frac{3}{2}\ln ^3\left(2\right)\zeta \left(2\right).\tag{2.15}
\label{2.x11}
\end{equation*}
For the second, by applying the substitution $\arctan \left(x\right)\mapsto x$, we obtain
$$\int _0^{\infty }\frac{\arctan ^3\left(x\right)\ln \left(1+x^2\right)}{1+x^2}\:dx=-2\int _0^{\frac{\pi }{2}}x^3\ln \left(\cos \left(x\right)\right)\:dx.$$
Therefore, by using the result from point $\left(\operatorname{viii}\right)$ in Lemma \ref{lma3}, we arrive at
\begin{equation*}
\int _0^{\infty }\frac{\arctan ^3\left(x\right)\ln \left(1+x^2\right)}{1+x^2}\:dx=-\frac{93}{64}\zeta \left(5\right)+\frac{9}{4}\zeta \left(2\right)\zeta \left(3\right)+\frac{45}{16}\ln \left(2\right)\zeta \left(4\right).\tag{2.16}
\label{2.x12}
\end{equation*}
For the third, we first evaluate the integral and then take the imaginary part. Consequently, if we split at $x=1$ and employ the substitution $\frac{1}{1+ix}\mapsto x$, we get
$$\int _0^{\infty }\frac{\ln ^4\left(1+ix\right)}{1+x^2}\:dx=i\int _0^{\frac{1-i}{2}}\frac{\ln ^4\left(x\right)}{1-2x}\:dx+i\int _{\frac{1-i}{2}}^1\frac{\ln ^4\left(x\right)}{1-2x}\:dx,$$
and by repeatedly applying integration by parts to evaluate the expressions, we find
$$\int _0^{\infty }\frac{\ln ^4\left(1+ix\right)}{1+x^2}\:dx=i\:12\operatorname{Li}_5\left(2\right).$$
Thus, if we use the identity $\operatorname{Li}_5\left(2\right)=\operatorname{Li}_5\left(\frac{1}{2}\right)+2\ln \left(2\right)\zeta \left(4\right)+\frac{1}{3}\ln ^3\left(2\right)\zeta \left(2\right)-\frac{1}{120}\ln ^5\left(2\right)-i\frac{\pi }{24}\ln ^4\left(2\right)$ (see \cite[(1.115), p. 47]{book1}) and then take the imaginary part, we obtain
\begin{equation*}
\operatorname{\mathfrak{I}} \left\{\int _0^{\infty }\frac{\ln ^4\left(1+ix\right)}{1+x^2}\:dx\right\}=12\operatorname{Li}_5\left(\frac{1}{2}\right)+24\ln \left(2\right)\zeta \left(4\right)+4\ln ^3\left(2\right)\zeta \left(2\right)-\frac{1}{10}\ln ^5\left(2\right).\tag{2.17}
\label{2.x13}
\end{equation*}
Hence, by applying $\left(\ref{2.x11}\right)$, $\left(\ref{2.x12}\right)$, and $\left(\ref{2.x13}\right)$ to $\left(\ref{2.x10}\right)$, we arrive at the desired result for the integral in point $\left(\operatorname{ii}\right)$.
\end{proof}
\end{lemma}

\begin{lemma} \label{lma8}The following equalities hold:
\begin{alignat*}{1}
\left(\operatorname{i}\right)\quad&\int _0^{\frac{\pi }{2}}x\ln \left(\sin \left(x\right)\right)\ln ^2\left(\cos \left(x\right)\right)\:dx\\
&=\frac{155}{128}\zeta \left(5\right)+\frac{13}{32}\zeta \left(2\right)\zeta \left(3\right)-\operatorname{Li}_5\left(\frac{1}{2}\right)-\frac{49}{32}\ln \left(2\right)\zeta \left(4\right)-\frac{5}{6}\ln ^3\left(2\right)\zeta \left(2\right)+\frac{1}{120}\ln ^5\left(2\right);\\
\left(\operatorname{ii}\right)\quad&\int _0^{\frac{\pi }{2}}x\ln ^2\left(\sin \left(x\right)\right)\ln \left(\cos \left(x\right)\right)\:dx\\
&=-\frac{155}{128}\zeta \left(5\right)-\frac{1}{32}\zeta \left(2\right)\zeta \left(3\right)+\operatorname{Li}_5\left(\frac{1}{2}\right)+\frac{49}{32}\ln \left(2\right)\zeta \left(4\right)-\frac{2}{3}\ln ^3\left(2\right)\zeta \left(2\right)-\frac{1}{120}\ln ^5\left(2\right).
\end{alignat*}
\begin{proof}
We begin by denoting the integral in point $\left(\operatorname{i}\right)$ as $I$ and the integral in point $\left(\operatorname{ii}\right)$ as $J$. Taking their difference yields
\begin{equation*}
I-J=\int _0^{\frac{\pi }{2}}x\ln \left(\sin \left(x\right)\right)\ln ^2\left(\cos \left(x\right)\right)\:dx-\int _0^{\frac{\pi }{2}}x\ln ^2\left(\sin \left(x\right)\right)\ln \left(\cos \left(x\right)\right)\:dx.\tag{2.18}
\label{2.x1}
\end{equation*}
By using the identity $ab^2-a^2b=\frac{1}{3}\left(a-b\right)^3-\frac{1}{3}a^3+\frac{1}{3}b^3$, with $a=\ln \left(\sin \left(x\right)\right)$ and $b=\ln \left(\cos \left(x\right)\right)$, we have
$$I-J=\frac{1}{3}\int _0^{\frac{\pi }{2}}x\ln ^3\left(\tan \left(x\right)\right)\:dx-\frac{1}{3}\int _0^{\frac{\pi }{2}}x\ln ^3\left(\sin \left(x\right)\right)\:dx+\frac{1}{3}\int _0^{\frac{\pi }{2}}x\ln ^3\left(\cos \left(x\right)\right)\:dx,$$
and if we apply the substitution $x\mapsto \frac{\pi }{2}-x$ to the rightmost integral, we get
$$I-J=\frac{1}{3}\int _0^{\frac{\pi }{2}}x\ln ^3\left(\tan \left(x\right)\right)\:dx-\frac{2}{3}\int _0^{\frac{\pi }{2}}x\ln ^3\left(\sin \left(x\right)\right)\:dx+\frac{\pi }{6}\int _0^{\frac{\pi }{2}}\ln ^3\left(\sin \left(x\right)\right)\:dx.$$
Therefore, by using the results from points $\left(\operatorname{i}\right)$ and $\left(\operatorname{iii}\right)$ in Lemma \ref{lma5}, and from point $\left(\operatorname{ii}\right)$ in Lemma \ref{lma7}, we obtain
\begin{equation*}
I-J=\frac{155}{64}\zeta \left(5\right)+\frac{7}{16}\zeta \left(2\right)\zeta \left(3\right)-2\operatorname{Li}_5\left(\frac{1}{2}\right)-\frac{49}{16}\ln \left(2\right)\zeta \left(4\right)-\frac{1}{6}\ln ^3\left(2\right)\zeta \left(2\right)+\frac{1}{60}\ln ^5\left(2\right).\tag{2.19}
\label{2.x2}
\end{equation*}
Furthermore, if we apply the substitution $x\mapsto \frac{\pi }{2}-x$ to $J$ in $\left(\ref{2.x1}\right)$, we have
$$I-J=2\int _0^{\frac{\pi }{2}}x\ln \left(\sin \left(x\right)\right)\ln ^2\left(\cos \left(x\right)\right)\:dx-\frac{\pi }{2}\int _0^{\frac{\pi }{2}}\ln \left(\sin \left(x\right)\right)\ln ^2\left(\cos \left(x\right)\right)\:dx,$$
and since $2I$ appears on the right-hand side, along with the integral from point $\left(\operatorname{ii}\right)$ in Lemma \ref{lma5}, it follows that
\begin{equation*}
I+J=\frac{3}{8}\zeta \left(2\right)\zeta \left(3\right)-\frac{3}{2}\ln ^3\left(2\right)\zeta \left(2\right).\tag{2.20}
\label{2.x3}
\end{equation*}
Thus, by solving for $I$ and $J$ using $\left(\ref{2.x2}\right)$ and $\left(\ref{2.x3}\right)$, we obtain the desired closed forms.
\end{proof}
\end{lemma}

\section{The Main Series and their Proofs}
\begin{theorem}The following equalities hold:
\begin{alignat*}{1}
\left(\operatorname{i}\right)\quad&\sum _{k=1}^{\infty }\frac{4^kH_k^{\left(3\right)}}{k^2\binom{2k}{k}}\\
&=\frac{217}{8}\zeta \left(5\right)-\frac{9}{2}\zeta \left(2\right)\zeta \left(3\right)-16\operatorname{Li}_5\left(\frac{1}{2}\right)-\frac{19}{2}\ln \left(2\right)\zeta \left(4\right)+\frac{8}{3}\ln ^3\left(2\right)\zeta \left(2\right)+\frac{2}{15}\ln ^5\left(2\right);\\
\left(\operatorname{ii}\right)\quad&\sum _{k=1}^{\infty }\frac{4^kH_k^{\left(2\right)}}{k^3\binom{2k}{k}}\\
&=\frac{31}{4}\zeta \left(5\right)+\frac{3}{2}\zeta \left(2\right)\zeta \left(3\right)-16\operatorname{Li}_5\left(\frac{1}{2}\right)-2\ln \left(2\right)\zeta \left(4\right)+\frac{8}{3}\ln ^3\left(2\right)\zeta \left(2\right)+\frac{2}{15}\ln ^5\left(2\right);\\
\left(\operatorname{iii}\right)\quad&\sum _{k=1}^{\infty }\frac{4^kH_k^2}{k^3\binom{2k}{k}}\\
&=-62\zeta \left(5\right)-\frac{21}{2}\zeta \left(2\right)\zeta \left(3\right)+48\operatorname{Li}_5\left(\frac{1}{2}\right)+81\ln \left(2\right)\zeta \left(4\right)+8\ln ^3\left(2\right)\zeta \left(2\right)-\frac{2}{5}\ln ^5\left(2\right);\\
\left(\operatorname{iv}\right)\quad&\sum _{k=1}^{\infty }\frac{4^kH_k}{k^4\binom{2k}{k}}\\
&=-\frac{31}{2}\zeta \left(5\right)+3\zeta \left(2\right)\zeta \left(3\right)+16\operatorname{Li}_5\left(\frac{1}{2}\right)+2\ln \left(2\right)\zeta \left(4\right)+\frac{16}{3}\ln ^3\left(2\right)\zeta \left(2\right)-\frac{2}{15}\ln ^5\left(2\right);\\
\left(\operatorname{v}\right)\quad&\sum _{k=1}^{\infty }\frac{4^kH_kH_k^{\left(2\right)}}{k^2\binom{2k}{k}}\\
&=-\frac{155}{8}\zeta \left(5\right)+9\zeta \left(2\right)\zeta \left(3\right)+16\operatorname{Li}_5\left(\frac{1}{2}\right)+\frac{19}{2}\ln \left(2\right)\zeta \left(4\right)+\frac{16}{3}\ln ^3\left(2\right)\zeta \left(2\right)-\frac{2}{15}\ln ^5\left(2\right);\\
\left(\operatorname{vi}\right)\quad&\sum _{k=1}^{\infty }\frac{4^kH_k^3}{k^2\binom{2k}{k}}\\
&=-\frac{155}{8}\zeta \left(5\right)+18\zeta \left(2\right)\zeta \left(3\right)+80\operatorname{Li}_5\left(\frac{1}{2}\right)+\frac{455}{2}\ln \left(2\right)\zeta \left(4\right)+\frac{32}{3}\ln ^3\left(2\right)\zeta \left(2\right)-\frac{2}{3}\ln ^5\left(2\right).
\end{alignat*}
\end{theorem}

\begin{proof}
For the series in point $\left(\operatorname{i}\right)$, we begin by using the identity in $\left(\ref{1.3}\right)$. By letting $m=2$ and $z\mapsto k$, we obtain
$$\psi ^{\left(2\right)}\left(k+1\right)=-2\zeta \left(3\right)+2H_k^{\left(3\right)},$$
and if we isolate $H_k^{\left(3\right)}$, multiply both sides by $\frac{4^k}{k^2\binom{2k}{k}}$, and sum from $k=1$ to $\infty$, we get
$$\sum _{k=1}^{\infty }\frac{4^kH_k^{\left(3\right)}}{k^2\binom{2k}{k}}=\zeta \left(3\right)\sum _{k=1}^{\infty }\frac{4^k}{k^2\binom{2k}{k}}+\frac{1}{2}\sum _{k=1}^{\infty }\frac{4^k}{k^2\binom{2k}{k}}\psi ^{\left(2\right)}\left(k+1\right).$$
By noting that the first series on the right-hand side is the result from point $\left(\operatorname{i}\right)$ in Lemma \ref{lma4}, and applying the integral definition in $\left(\ref{1.2}\right)$ for $\psi ^{\left(2\right)}\left(k+1\right)$, we arrive at
$$\sum _{k=1}^{\infty }\frac{4^kH_k^{\left(3\right)}}{k^2\binom{2k}{k}}=3\zeta \left(2\right)\zeta \left(3\right)-\frac{1}{2}\int _0^1\left(\sum _{k=1}^{\infty }\frac{4^kx^k}{k^2\binom{2k}{k}}\right)\frac{\ln ^2\left(x\right)}{1-x}\:dx.$$
Substituting $x\mapsto x^2$, applying the result from point $\left(\operatorname{ii}\right)$ in Lemma \ref{lma1}, and then substituting $\arcsin \left(x\right)\mapsto x$, leads to
$$\sum _{k=1}^{\infty }\frac{4^kH_k^{\left(3\right)}}{k^2\binom{2k}{k}}=3\zeta \left(2\right)\zeta \left(3\right)-8\int _0^{\frac{\pi }{2}}x^2\tan \left(x\right)\ln ^2\left(\sin \left(x\right)\right)\:dx,$$
and if we perform the change of variables $x\mapsto \frac{\pi }{2}-x$, we obtain
$$\sum _{k=1}^{\infty }\frac{4^kH_k^{\left(3\right)}}{k^2\binom{2k}{k}}=3\zeta \left(2\right)\zeta \left(3\right)-12\zeta \left(2\right)\int _0^{\frac{\pi }{2}}\cot \left(x\right)\ln ^2\left(\cos \left(x\right)\right)\:dx+8\pi \int _0^{\frac{\pi }{2}}x\cot \left(x\right)\ln ^2\left(\cos \left(x\right)\right)\:dx$$
$$-8\int _0^{\frac{\pi }{2}}x^2\cot \left(x\right)\ln ^2\left(\cos \left(x\right)\right)\:dx.$$
Thus, by applying point $\left(\operatorname{iv}\right)$ from Lemma \ref{lma4}, point $\left(\operatorname{ii}\right)$ from Lemma \ref{lma6}, and point $\left(\operatorname{i}\right)$ from Lemma \ref{lma7}, we obtain the desired result for the series in point $\left(\operatorname{i}\right)$.
\\[1.5ex]
For the series in point $\left(\operatorname{ii}\right)$, we start with the generating function from point $\left(\operatorname{iii}\right)$ in Lemma \ref{lma1}, where multiplying both sides by $\frac{1}{x}$ and integrating from $x=0$ to $x=1$ yields
\begin{equation*}
\frac{3}{4}\sum _{k=1}^{\infty }\frac{4^kH_k^{\left(2\right)}}{k^3\binom{2k}{k}}-\frac{3}{4}\sum _{k=1}^{\infty }\frac{4^k}{k^5\binom{2k}{k}}=\int _0^1\frac{\arcsin ^4\left(x\right)}{x}\:dx.\tag{3.1}
\label{3.1}
\end{equation*}
For the expression on the right-hand side of $\left(\ref{3.1}\right)$, by integrating by parts and then using the substitution $\arcsin \left(x\right)\mapsto x$, we get
\begin{equation*}
\int _0^1\frac{\arcsin ^4\left(x\right)}{x}\:dx=-4\int _0^{\frac{\pi }{2}}x^3\ln \left(\sin \left(x\right)\right)\:dx.\tag{3.2}
\label{3.2}
\end{equation*}
Furthermore, for the second series on the left-hand side of $\left(\ref{3.1}\right)$, we have
$$\sum _{k=1}^{\infty \:}\frac{4^k}{k^5\binom{2k}{k}}=4\int _0^1\left(\sum _{k=1}^{\infty }\frac{4^kx^{2k}}{k^2\binom{2k}{k}}\right)\frac{\ln ^2\left(x\right)}{x}\:dx,$$
and if we apply the result from point $\left(\operatorname{ii}\right)$ in Lemma \ref{lma1}, integrate by parts, and substitute $\arcsin \left(x\right)\mapsto x$, we arrive at
\begin{equation*}
\sum _{k=1}^{\infty }\frac{4^k}{k^5\binom{2k}{k}}=-\frac{16}{3}\int _0^{\frac{\pi }{2}}x\ln ^3\left(\sin \left(x\right)\right)\:dx.\tag{3.3}
\label{3.3}
\end{equation*}
By applying $\left(\ref{3.2}\right)$ and $\left(\ref{3.3}\right)$ to $\left(\ref{3.1}\right)$, and then isolating the target series, we obtain
$$\sum _{k=1}^{\infty }\frac{4^kH_k^{\left(2\right)}}{k^3\binom{2k}{k}}=-\frac{16}{3}\int _0^{\frac{\pi }{2}}x^3\ln \left(\sin \left(x\right)\right)\:dx-\frac{16}{3}\int _0^{\frac{\pi }{2}}x\ln ^3\left(\sin \left(x\right)\right)\:dx,$$
and if we employ the result from point $\left(\operatorname{vii}\right)$ in Lemma \ref{lma3} and the result from point $\left(\operatorname{ii}\right)$ in Lemma \ref{lma7}, we obtain the desired closed form for the series in point $\left(\operatorname{ii}\right)$.
\\[1.5ex]
For the series in point $\left(\operatorname{iii}\right)$, we utilize the result from point $\left(\operatorname{iii}\right)$ in Lemma \ref{lma2}, where, by multiplying both sides by $\frac{4^k}{k^2\binom{2k}{k}}$ and summing from $k=1$ to $\infty$, we obtain
$$\int _0^1\left(\sum _{k=1}^{\infty }\frac{4^kx^k}{k^2\binom{2k}{k}}\right)\frac{\ln ^2\left(1-x\right)}{x}\:dx=\sum _{k=1}^{\infty }\frac{4^kH_k^2}{k^3\binom{2k}{k}}+\sum _{k=1}^{\infty }\frac{4^kH_k^{\left(2\right)}}{k^3\binom{2k}{k}}.$$
If we use the substitution $x\mapsto x^2$, apply the result from point $\left(\operatorname{ii}\right)$ in Lemma \ref{lma1}, followed by employing the substitution $\arcsin \left(x\right)\mapsto x$, and then isolate the desired series, we arrive at
$$\sum _{k=1}^{\infty }\frac{4^kH_k^2}{k^3\binom{2k}{k}}=16\int _0^{\frac{\pi }{2}}x^2\cot \left(x\right)\ln ^2\left(\cos \left(x\right)\right)\:dx-\sum _{k=1}^{\infty }\frac{4^kH_k^{\left(2\right)}}{k^3\binom{2k}{k}},$$
and by using the result of the series in point $\left(\operatorname{ii}\right)$ and the result from point $\left(\operatorname{i}\right)$ in Lemma \ref{lma7}, the announced closed form for the series in point $\left(\operatorname{iii}\right)$ follows.
\\[1.5ex]
For the series in point $\left(\operatorname{iv}\right)$, we proceed by using the result from point $\left(\operatorname{ii}\right)$ in Lemma \ref{lma2}, if we multiply both sides by $\frac{4^k}{k\binom{2k}{k}}$ and sum from $k=1$ to $\infty$, we get
$$\int _0^1\left(\sum _{k=1}^{\infty }\frac{4^kx^k}{k\binom{2k}{k}}\right)\frac{\ln ^2\left(x\right)\ln \left(1-x\right)}{x}\:dx=-2\sum _{k=1}^{\infty }\frac{4^kH_k}{k^4\binom{2k}{k}}-2\sum _{k=1}^{\infty }\frac{4^kH_k^{\left(2\right)}}{k^3\binom{2k}{k}}-2\sum _{k=1}^{\infty }\frac{4^kH_k^{\left(3\right)}}{k^2\binom{2k}{k}}$$
$$+2\zeta \left(2\right)\sum _{k=1}^{\infty }\frac{4^k}{k^3\binom{2k}{k}}+2\zeta \left(3\right)\sum _{k=1}^{\infty }\frac{4^k}{k^2\binom{2k}{k}}.$$
By employing the substitution $x\mapsto x^2$, applying the result from point $\left(\operatorname{i}\right)$ in Lemma \ref{lma1}, followed by the substitution $\arcsin \left(x\right)\mapsto x$, and afterward isolating the target series, we arrive at
$$\sum _{k=1}^{\infty }\frac{4^kH_k}{k^4\binom{2k}{k}}=-16\int _0^{\frac{\pi }{2}}x\ln ^2\left(\sin \left(x\right)\right)\ln \left(\cos \left(x\right)\right)\:dx-\sum _{k=1}^{\infty }\frac{4^kH_k^{\left(2\right)}}{k^3\binom{2k}{k}}-\sum _{k=1}^{\infty }\frac{4^kH_k^{\left(3\right)}}{k^2\binom{2k}{k}}$$
$$+\zeta \left(2\right)\sum _{k=1}^{\infty }\frac{4^k}{k^3\binom{2k}{k}}+\zeta \left(3\right)\sum _{k=1}^{\infty }\frac{4^k}{k^2\binom{2k}{k}}.$$
After using the result from point $\left(\operatorname{ii}\right)$ in Lemma \ref{lma8}, the result of the series in points $\left(\operatorname{ii}\right)$ and $\left(\operatorname{i}\right)$, and the results of the series in points $\left(\operatorname{ii}\right)$ and $\left(\operatorname{i}\right)$ in Lemma \ref{lma4}, we arrive at the closed form of the series in point $\left(\operatorname{iv}\right)$.
\\[1.5ex]
For the series in point $\left(\operatorname{v}\right)$, we commence by using the result from point $\left(\operatorname{iv}\right)$ in Lemma \ref{lma2}. If we multiply both sides by $\frac{4^k}{k\binom{2k}{k}}$ and sum from $k=1$ to $\infty$, we obtain
$$\int _0^1\left(\sum _{k=1}^{\infty }\frac{4^kx^{k-1}}{k\binom{2k}{k}}\right)\ln \left(x\right)\ln ^2\left(1-x\right)\:dx=-\sum _{k=1}^{\infty }\frac{4^kH_k^2}{k^3\binom{2k}{k}}-\sum _{k=1}^{\infty }\frac{4^kH_k^{\left(2\right)}}{k^3\binom{2k}{k}}-2\sum _{k=1}^{\infty }\frac{4^kH_k^{\left(3\right)}}{k^2\binom{2k}{k}}$$
$$-2\sum _{k=1}^{\infty }\frac{4^kH_kH_k^{\left(2\right)}}{k^2\binom{2k}{k}}+2\zeta \left(2\right)\sum _{k=1}^{\infty }\frac{4^kH_k}{k^2\binom{2k}{k}}+2\zeta \left(3\right)\sum _{k=1}^{\infty }\frac{4^k}{k^2\binom{2k}{k}},$$
and by using the substitution $x\mapsto x^2$, applying the result from point $\left(\operatorname{i}\right)$ in Lemma \ref{lma1}, then substituting $\arcsin \left(x\right)\mapsto x$ and isolating the desired series, we get
$$\sum _{k=1}^{\infty }\frac{4^kH_kH_k^{\left(2\right)}}{k^2\binom{2k}{k}}=-16\int _0^{\frac{\pi }{2}}x\ln \left(\sin \left(x\right)\right)\ln ^2\left(\cos \left(x\right)\right)\:dx-\frac{1}{2}\sum _{k=1}^{\infty }\frac{4^kH_k^2}{k^3\binom{2k}{k}}-\frac{1}{2}\sum _{k=1}^{\infty }\frac{4^kH_k^{\left(2\right)}}{k^3\binom{2k}{k}}$$
$$-\sum _{k=1}^{\infty }\frac{4^kH_k^{\left(3\right)}}{k^2\binom{2k}{k}}+\zeta \left(2\right)\sum _{k=1}^{\infty }\frac{4^kH_k}{k^2\binom{2k}{k}}+\zeta \left(3\right)\sum _{k=1}^{\infty }\frac{4^k}{k^2\binom{2k}{k}}.$$
Therefore, if we use the result from point $\left(\operatorname{i}\right)$ in Lemma \ref{lma8}, the closed forms of the series in points $\left(\operatorname{iii}\right)$, $\left(\operatorname{ii}\right)$, and $\left(\operatorname{i}\right)$, and subsequently apply the results from points $\left(\operatorname{iii}\right)$ and $\left(\operatorname{i}\right)$ in Lemma \ref{lma4}, we arrive at the desired result for the series in point $\left(\operatorname{v}\right)$.
\\[1.5ex]
For the series in point $\left(\operatorname{vi}\right)$, we will use the result from point $\left(\operatorname{v}\right)$ in Lemma \ref{lma2}. If we multiply both sides by $\frac{4^k}{k\binom{2k}{k}}$ and sum from $k=1$ to $\infty$, we get
$$\int _0^1\left(\sum _{k=1}^{\infty }\frac{4^kx^{k-1}}{k\binom{2k}{k}}\right)\ln ^3\left(1-x\right)\:dx=-\sum _{k=1}^{\infty }\frac{4^kH_k^3}{k^2\binom{2k}{k}}-3\sum _{k=1}^{\infty }\frac{4^kH_kH_k^{\left(2\right)}}{k^2\binom{2k}{k}}-2\sum _{k=1}^{\infty }\frac{4^kH_k^{\left(3\right)}}{k^2\binom{2k}{k}},$$
and by substituting $x\mapsto x^2$, applying the result from point $\left(\operatorname{i}\right)$ in Lemma \ref{lma1}, and then substituting $\arcsin \left(x\right)\mapsto x$, we obtain
$$32\int _0^{\frac{\pi }{2}}x\ln ^3\left(\cos \left(x\right)\right)\:dx=-\sum _{k=1}^{\infty }\frac{4^kH_k^3}{k^2\binom{2k}{k}}-3\sum _{k=1}^{\infty }\frac{4^kH_kH_k^{\left(2\right)}}{k^2\binom{2k}{k}}-2\sum _{k=1}^{\infty }\frac{4^kH_k^{\left(3\right)}}{k^2\binom{2k}{k}}.$$
If we use the substitution $x\mapsto \frac{\pi }{2}-x$ and isolate the desired series, we arrive at
$$\sum _{k=1}^{\infty }\frac{4^kH_k^3}{k^2\binom{2k}{k}}=32\int _0^{\frac{\pi }{2}}x\ln ^3\left(\sin \left(x\right)\right)\:dx-16\pi \int _0^{\frac{\pi }{2}}\ln ^3\left(\sin \left(x\right)\right)\:dx-3\sum _{k=1}^{\infty }\frac{4^kH_kH_k^{\left(2\right)}}{k^2\binom{2k}{k}}-2\sum _{k=1}^{\infty }\frac{4^kH_k^{\left(3\right)}}{k^2\binom{2k}{k}},$$
and if we utilize the result from point $\left(\operatorname{ii}\right)$ in Lemma \ref{lma7}, the result from point $\left(\operatorname{i}\right)$ in Lemma \ref{lma5}, and the results of the series in points $\left(\operatorname{v}\right)$ and $\left(\operatorname{i}\right)$, we obtain the closed form for the series in point $\left(\operatorname{vi}\right)$.
\end{proof}

\bibliographystyle{amsplain}

\end{document}